\documentclass[11pt]{amsart}    
\usepackage{graphicx}
\usepackage{color}
\usepackage{hyperref}
\usepackage{geometry}
\usepackage{txfonts}
\usepackage{color}
\definecolor{Myred}{cmyk}{0.0,1.0,1.0,0.00}
\definecolor{Mypurple}{rgb}{0.5,0.0,0.5}
\newtheorem{theorem}{Theorem}

\usepackage{enumerate}

\begin{document}

\title[Spectral properties of the magnetic Robin Laplacian on curvilinear half-plane]
{Spectral properties of the magnetic Robin Laplacian on curvilinear half-plane}

\author{Diana Barseghyan (Schneiderov\'{a})$^\ast$}\thanks{$^\ast$Corresponding author.}
\address{Department of Mathematics, University of Ostrava, 30. dubna 22, 70103 Ostrava, Czech Republic}
\email{diana.schneiderova@osu.cz}

\author{Baruch Schneider}

\address{Department of Mathematics, University of Ostrava, 30. dubna 22, 70103 Ostrava, Czech Republic}
\email{baruch.schneider@osu.cz}

\author{Yifan Zhang}
\address
{Department of Mathematics, University of Ostrava, Ostrava, Czech Republic\\
Department of Applied Mathematics, VSB – Technical University of Ostrava, Ostrava, Czech Republic\\
Department of Algebra, Charles University, Prague, Czech Republic}
\email{yifan.zhang@osu.cz}



\maketitle

\begin{abstract}
We analyse magnetic Robin Laplacian in curvilinear half-plane with a smooth boundary. It is well known that the
spectrum of the Robin Laplacian is unstable with respect to boundary deformations. This means that if the boundary is a straight line then the spectrum of the Robin Laplacian is purely essential. From the other hand, the perturbation of the boundary produces eigenvalues below the essential spectrum.
In this paper, the Robin-Laplace operator with a compactly supported magnetic field is considered.
We prove that the spectrum of the magnetic Robin Laplacian is stable under small and local deformations of the boundary.
\end{abstract}
\bigskip

\bigskip

\section{Introduction} \label{sect: intro}

As is well known, bending a two-dimensional quantum waveguide in the appropriate way induces bound states (see references \cite{DE95, ES89}, and \cite{GJ92}). Mathematically speaking, we can say that the Dirichlet
Laplacian on a smooth, asymptotically straight planar waveguide has at least
one isolated eigenvalue below the threshold of the essential spectrum.
Similar results have been obtained for a locally deformed waveguide, which  is equivalent 
to adding a small ``bump'' to a straight waveguide (see references \cite{BEGK01} and 
\cite{BGRS97}). Consequently, for any nonzero curvature satisfying certain regularity properties, at least one isolated eigenvalue appears below the essential spectrum.

Recall that a magnetic field, even a local one, can significantly affect the behaviour of waveguide systems, particularly the existence of a geometrically induced discrete spectrum. While a particle confined to a fixed-profile, asymptotically straight tube with a hard-wall boundary can have localized states whenever the tube is bent or deformed locally  (see \cite{EK15} for a comprehensive review of quantum waveguide theory), a local magnetic field can destroy this discrete spectrum.

A similar, in which the existence of bound states results from the geometry of the interaction support, has been observed for a class of singular Schr\"{o}dinger operators usually called leaky quantum wires, with an attractive contact interaction supported by a curve \cite{EK15}. As established in \cite{BE21} a local magnetic field can destroy the discrete spectrum again. A related result was obtained in \cite{BBS24}, where the authors considered a magnetic Schr\"{o}dinger operator with a non-negative potential supported on a strip that is a local deformation of a straight one, while the magnetic field was assumed to be nonzero and local.
 
In this work we consider another problem in which the region is a curvilinear half- plane and its boundary is described by a mixed-type condition, conventionally called Robin, with the parameter $\beta>0$ and the magnetic field is supposed to be local.
The result for a single infinite boundary curve naturally raises the question of the existence of bound states. The answer to this question is given in work \cite{EM14} where the authors established that if the curve is a nontrivial local deformation of the straight line, then the essential spectrum of the Robin Laplacian coincides with the half-line $[-\beta^2, \infty)$ and the discrete spectrum below $-\beta^2$ is non-empty.

The domain $\Omega$ studied in our paper is assumed to be a curvilinear upper half-plane with the boundary  $\Gamma = \{(a(s), b(s)): s\in\mathbb{R}\}$, where $a, b$ are sufficiently smooth functions satisfying
\begin{equation}\label{ab}
\dot{a}(s)^2 + \dot{b}(s)^2 = 1,
\end{equation} 
where dot marks the derivative with respect to $s$, so $s$ is the arc length of $\Gamma$. It is useful to introduce the signed curvature $\gamma(s)$ of $\Gamma$,
\begin{equation}
\label{curv.}
\gamma(s)=\dot{b}(s)\ddot{a}(s)- \dot{a}(s)\ddot{b}(s).
\end{equation}

We first derive the following identity, which will be used later. From (\ref{ab}) and (\ref{curv.}), we obtain
\begin{eqnarray}\nonumber\label{gamma} \gamma(s)^2= \left(\dot{b}(s)\ddot{a}(s)- \dot{a}(s)\ddot{b}(s)\right)^2\\\label{aux}=
\ddot{a}(s)^2- \ddot{a}(s)^2 \dot{a}(s)^2 + \ddot{b}(s)^2\dot{a}(s)^2- 2\ddot{a}(s)\dot{b}(s)\ddot{b}(s)\dot{a}(s).
\end{eqnarray}

With respect to identity (\ref{ab}), we have 
\begin{equation}\label{adot}
\dot{a}(s)\ddot{a}(s)+\dot{b}(s)\ddot{b}(s)=0.
\end{equation}
Therefore, (\ref{aux}) implies 
\begin{eqnarray}\nonumber\gamma(s)^2=
\ddot{a}(s)^2- \ddot{b}(s)^2 \dot{b}(s)^2 + \ddot{b}(s)^2\dot{a}(s)^2+ 2\ddot{b}(s)^2\dot{b}(s)^2= \\\label{id}
=\ddot{a}(s)^2+\ddot{b}(s)^2.
\end{eqnarray}

In fact, it is sufficient to know the function $\gamma$ only, since $a, b$ can be reconstructed from the relations
\begin{eqnarray}\nonumber 
a(s) = a(s_0) + \int_{s_0}^s \cos \left(\int_{s_0}^\eta \gamma(\theta)\,d \theta\right)\,d \eta\,,\\ \label{reconstr.}
b(s) = b(s_0) + \int_{s_0}^s \sin \left(\int_{s_0}^\eta \gamma(\theta)\,d \theta\right)\,d \eta\,,
\end{eqnarray}
where $s_0$ is a fixed number.

The object of our interest in this paper is the magnetic Robin Laplacian $H^\mathcal{A}_{\beta, \Omega}$ which is the unique self-adjoint operator associated with the following closed and below bounded quadratic form 
\begin{equation}\label{differential form}
q^\mathcal{A}_{\beta, \Omega}[\varphi] =\|(i\nabla+ \mathcal{A})\varphi\|^2_{L^2(\Omega)}- \beta\int_\Gamma|\varphi|^2\,d s\,,
\end{equation}
where $\beta>0$ is a fixed number, $\mathcal{A}$ is a magnetic potential supposed to be compactly supported and $\mathrm{Dom}(q^\mathcal{A}_{\beta, \Omega}) = \mathcal{H}^{1}(\Omega)$. 

During our paper we assume the following conditions: 
\begin{enumerate}[(a)]
\setlength{\itemsep}{3pt}

\item The graph of curve $\Gamma$ for $x\in (-\alpha, \alpha),\, \alpha>0$ is sufficiently smooth and for $x\in (-\infty, -\alpha+ \delta)\cup (\alpha- \delta, \infty)$ with some $0< \delta< \alpha$ coincides with the coordinate axis $\{x, 0\}_{x\in \mathbb{R}}$. 

\item
We are interested in magnetic fields
\begin{equation}\label{AA0}\mathcal{A}(x, y)= \mathcal{B} \mathcal{A}_0(x, y)\,,
\end{equation} 
where $\mathcal{B}> 0$ and
\begin{eqnarray}\nonumber
\mathcal{A}_0 =(0, a_2^0(x))\,,\\\label{A0}
a_2^0(x)= x,\,\text{ if}\,x\in (-\alpha, \alpha)\,\text{ and is vanishing at infinity}. 
\end{eqnarray}
\end{enumerate}

The similar operator but with zero magnetic field was considered in work \cite{EM14}. In this work it was proved that if $\Gamma$ is a local nontrivial perturbation of the straight line then the essential spectrum of operator 
$H^0_{\beta, \Omega}$ covers the half-line $[-\beta^2, \infty)$ but the discrete spectrum below the threshold of the essential spectrum is always non-empty. We will show that the suitable compactly supported magnetic field could destroy the discrete spectrum but preserves the essential spectrum.

\section{Emptiness of the discrete spectrum}
\label{emptiness}
\setcounter{equation}{0}

We will use the notation $\|\cdot\|_\infty= \|\cdot\|_{L^\infty(\mathbb{R})}$. The first main result of this paper is the following theorem:
 
\begin{theorem}\label{thm:dia-mag}
Suppose the validity of the assumptions (a)- (b) stated in Introduction. Let $0<c_1<c_2$ and $c_1\beta^\sigma\le \mathcal{B}\le c_2\beta^\sigma,\,\sigma\ge 2$. There exist positive numbers $\beta_0$ and $\varepsilon_0(\mathcal{B})$ such that for all
$\beta>\beta_0$ and under the condition $\mathrm{max}\left\{\|\gamma\|_\infty, \,\|\dot{\gamma}\|_\infty\right\}< \varepsilon_0(\mathcal{B})$
the discrete spectrum of operator $H^\mathcal{A}_{\beta, \Omega}$  below $-\beta^2$ is empty.
\end{theorem}

\begin{proof}

The quadratic form (\ref{differential form}) can be rewritten as follows
\begin{eqnarray}\nonumber
q_{\beta, \Omega}^\mathcal{A}(\varphi)= \int_\Omega\left(\left|\frac{\partial\varphi}{\partial x}\right|^2+\left|i\frac{\partial\varphi}{\partial y}+ \mathcal{B} a^0_2\varphi\right|^2\right)\,d x\,d y- \beta\int_\Gamma |\varphi|^2\,d s\\\nonumber
= \int_{\Omega\cap \{x\in (-\alpha, \alpha)\}}\left(\left|\frac{\partial\varphi}{\partial x}\right|^2+\left|i\frac{\partial\varphi}{\partial y}+ \mathcal{B} a^0_2\varphi\right|^2\right)\,d x\,d y- \beta \int_{\Gamma\cap\{x\in (-\alpha, \alpha)\}} |\varphi|^2\,d s\\\label{qf}+  \int_{\Omega\cap\{x\in (-\infty, -\alpha)\cup (\alpha, \infty)\}}\left(\left|\frac{\partial\varphi}{\partial x}\right|^2+\left|i\frac{\partial\varphi}{\partial y}+ \mathcal{B} a^0_2\varphi\right|^2\right)\,d x\,d y- \beta \int_{\Gamma\cap\{x\in (-\infty, -\alpha)\cup (\alpha, \infty)\}} |\varphi|^2\,d s\,.
\end{eqnarray}

Let us estimate separately expressions 
$$\int_{\Omega\cap \{x\in (-\alpha, \alpha)\}}\left(\left|\frac{\partial\varphi}{\partial x}\right|^2+\left|i\frac{\partial\varphi}{\partial y}+ \mathcal{B} a^0_2\varphi\right|^2\right)\,d x\,d y- \beta \int_{\Gamma\cap\{x\in (-\alpha, \alpha)\}} |\varphi|^2\,d s$$ and 
$$\int_{\Omega\cap\{x\in (-\infty, -\alpha)\cup (\alpha, \infty)\}}\left(\left|\frac{\partial\varphi}{\partial x}\right|^2+\left|i\frac{\partial\varphi}{\partial y}+ \mathcal{B} a^0_2\varphi\right|^2\right)\,d x\,d y- \beta \int_{\Gamma\cap\{x\in (-\infty, -\alpha)\cup (\alpha, \infty)\}} |\varphi|^2\,d s.$$

We will start with the first one. We will construct a bounded curved strip $\tilde{\Omega}\subset \Omega\cap \{x\in (-\alpha, \alpha)\}$ such that for any function $g\in \mathcal{H}^{1}(\tilde{\Omega})$ and under the assumptions stated in the theorem

\begin{equation}\label{cond.}\boxed{
\int_{\tilde{\Omega}}\left(\left|\frac{\partial g}{\partial x}\right|^2+\left|i\frac{\partial g}{\partial y}+ \mathcal{B} a^0_2g\right|^2\right)\,d x\,d y- \beta \int_{\Gamma\cap\{x\in (-\alpha, \alpha)\}} |g|^2\,d s\ge 0\,.}
\end{equation}  

The above inequality together with non-negativeness of the expression \newline\noindent$\int_{\left(\Omega\cap \{x\in (-\alpha, \alpha)\}\right)\setminus\tilde{\Omega}}\left(\left|\frac{\partial g}{\partial x}\right|^2+\left|i\frac{\partial g}{\partial y}+ \mathcal{B} a^0_2g\right|^2\right)\,d x\,d y$ proves that 

\begin{equation}\label{first integral}
\int_{\Omega\cap \{x\in (-\alpha, \alpha)\}}\left(\left|\frac{\partial\varphi}{\partial x}\right|^2+\left|i\frac{\partial\varphi}{\partial y}+ \mathcal{B} a^0_2\varphi\right|^2\right)\,d x\,d y- \beta \int_{\Gamma\cap\{x\in (-\alpha, \alpha)\}} |\varphi|^2\,d s\ge 0\,.
\end{equation}

The construction of the domain $\tilde{\Omega}$ will be discussed in the next section.

Finally let us pass to expression 
$$\int_{\Omega\cap\{x\in (-\infty, -\alpha)\cup (\alpha, \infty)\}}\left(\left|\frac{\partial\varphi}{\partial x}\right|^2+\left|i\frac{\partial\varphi}{\partial y}+ \mathcal{B} a^0_2\varphi\right|^2\right)\,d x\,d y- \beta \int_{\Gamma\cap\{x\in (-\infty, -\alpha)\cup (\alpha, \infty)\}} |\varphi|^2\,d s.
$$

We have
\begin{eqnarray}\nonumber
\int_{\Omega\{x\in (-\infty, -\alpha)\cup (\alpha, \infty)\}}\left(\left|\frac{\partial\varphi}{\partial x}\right|^2+\left|i\frac{\partial\varphi}{\partial y}+ \mathcal{B} a^0_2\varphi\right|^2\right)\,d x\,d y- \beta \int_{\Gamma\cap\{x\in (-\infty, -\alpha)\cup (\alpha, \infty)\}} |\varphi|^2\,d s\\\nonumber =\int_{\{|x|> \alpha,\, y\ge 0\}}\left(\left|\frac{\partial \psi}{\partial x}\right|^2+ \left|\frac{\partial \psi}{\partial y}
+ \mathcal{B} a_2^0(x) \psi\right|^2\right)\,d x\,d y- \beta \int_{\{y= 0\}_{\{|x|> \alpha\}}} |\psi|^2\,d s\,\\\nonumber 
\ge \int_{\{|x|> \alpha,\, y\ge 0\}} \left|\frac{\partial \psi}{\partial y}
+ \mathcal{B} a_2^0(x) \psi\right|^2\,d x\,d y- \beta \int_{\{y= 0\}_{\{|x|> \alpha\}}} |\psi|^2\,d s\,.
\end{eqnarray}

In view of the gauge invariance the right-hand side of the above bound is estimated from below by the ground state eigenvalue of one dimensional operator Robin Laplacian with coefficient $\beta$ on half-line which is exactly $-\beta^2$. Combining this together with (\ref{first integral}) one establishes the theorem. \end{proof}

\section{Construction of domain $\tilde{\Omega}$}
\setcounter{equation}{0}

To describe the region $\tilde{\Omega}$ we employ the natural locally orthogonal coordinates, in analogy with the theory of quantum waveguides \cite{ES89} in such a way that
 \begin{equation} \label{Omega-2}
\tilde{\Omega} :=\{ (a(s)-u \dot b(s),b(s)+u \dot a(s)):\: s\in (s_1, s_2),\, 0< u <  f(s)\, \}\,,
 \end{equation}
where numbers $s_1< s_2$ satisfy that for all $s\in (s_1, s_2)$ the graph of curve $\Gamma$ belongs to $\Omega\cap \{x\in (-\alpha, \alpha)\}$ and outside of interval $(s_1, s_2)$ it coincides with the straight line $\{x, 0\}_{x\in \mathbb{R}}$. In view of our assumptions stated in Introduction one can choose a non-negative smooth function $f$ in such a way that domain $\tilde{\Omega}\subset\Omega\cap \{x\in (-\alpha, \alpha)\}$ and has $C^3$ smooth boundary.

Let us mention that in view of our construction the boundary of $\tilde{\Omega}$ contains the "perturbed" part of curve $\Gamma$. 

We shall restrict ourselves from the beginning by the requirement which we will use later
$$
 \|f \gamma\|_{L^\infty(s_1, s_2)}< 1\,,
$$
which guarantees that 
$$
1-  \|f \gamma\|_{L^\infty(s_1, s_2)}>0.
$$

Let us denote by $\tilde{H}^\mathcal{A}_{\beta, \tilde{\Omega}}$ the unique self-adjoint operator associated with the following closed and below bounded quadratic form 
\begin{equation}
\label{qtilde}
\tilde{q}_{\beta, \tilde{\Omega}}^\mathcal{A}(g)=\|(i\nabla+ \mathcal{A})g\|^2_{L^2(\tilde{\Omega})}- \beta\int_{\partial \tilde{\Omega}}|g|^2\,d s\,,
\end{equation}
where $\partial \tilde{\Omega}$ is the boundary of $\tilde{\Omega}$ and $\mathrm{Dom}(\tilde{q}_{\beta, \tilde{\Omega}}^\mathcal{A}) = \mathcal{H}^{1}(\tilde{\Omega})$.  

It is easy to notice that the following theorem establishes (\ref{cond.}):

\begin{theorem}\label{appendix}\setcounter{equation}{0}
Suppose the validity of the stated assumptions in Theorem \ref{thm:dia-mag}. There exist positive numbers $\beta_0$ and $\varepsilon_0(\alpha, \mathcal{B})$ such that for all $\beta>\beta_0$ and under the condition
$\mathrm{max}\left\{\|\gamma\|_\infty, \,\|\dot{\gamma}\|_\infty\right\}< \varepsilon_0(\mathcal{B})$
the operator $\tilde{H}^\mathcal{A}_{\beta, \tilde{\Omega}}$ is non-negative.
\end{theorem}

\section{Proof of Theorem\,(\ref{appendix})}
\setcounter{equation}{0}

We define the unitary operator
\begin{equation} \label{U1}
U : L^2(\tilde{\Omega}) \to L^2(\Omega_0), \quad \Omega_0= \{(s, u)_{s\in (s_1, s_2),\,u\in \left(0, f(s)\right)}\}, 
\end{equation}
which for any $g\in L^2(\tilde{\Omega})$ acts as
\begin{equation} \label{U2}
h(s, u):=(U g)(s, u) = \sqrt{1 + u \gamma(s)}\, g(a(s) - u \dot{b}(s), b(s) + u \dot{a}(s)).
\end{equation}

Let $\Gamma_-= \{u= 0)\}_{s\in (s_1, s_2)},\,\Gamma_+= \{u= f(s))\}_{s\in (s_1, s_2)},\, \Gamma_1=  \{s= s_1\}_{u\in (0, f(s))}$ and $\Gamma_2= \{s= s_2\}_{u\in (0, f(s))}$. Using the indentities (\ref{ab})- (\ref{id}) we can check that the Jacobian
\begin{equation}\label{part.}
\frac{\partial(x, y)}{\partial(s, u)}= 1+u \gamma(s)\,,
\end{equation}
and
\begin{eqnarray}\nonumber
\frac{\partial g}{\partial x} = (1 +u \gamma)^{-1} \left(\dot{a}\frac{\partial }{\partial s} - (\dot{b} +
u \ddot{a}) \frac{\partial}{\partial u}\right)\left(\frac{h}{\sqrt{1 + u
\gamma}} \right)\,, \\\nonumber \frac{\partial g}{\partial y} = (1 + u \gamma)^{-1}
\left( \dot{b} \frac{\partial}{\partial s} + (\dot{a}- u \ddot{b})  \frac{\partial}{\partial u}\right)\left(\frac{h}{\sqrt{1 + u
\gamma}} \right)\,,\\\nonumber\int_{\Gamma_-}|g|^2\,d s= \int_{\{u= 0\}_{s\in (s_1, s_2)}}|h|^2\,d s\,,
\\\nonumber\int_{\Gamma_+}|g|^2\,d s= \int_{\{u= f(s)\}_{s\in (s_1, s_2)}}\frac{1}{1+ f \gamma} \sqrt{\frac{(1+ f \gamma)^2+ \dot{f}^2}{1+ \dot{f}^2}}|h|^2\,d s\,,
\\\nonumber \int_{\Gamma_1}|g|^2\,d s= \int_{\{s= s_1\}_{u\in (0, f(s_1))}} |h|^2\,d u\,,\\\label{part.1} \int_{\Gamma_2}|g|^2\,d s= \int_{\{s= s_2\}_{u\in (0, f(s_2))}} |h|^2\,d u\,.
\end{eqnarray}

Then, rewriting $a_2^0$ in coordinates $(s, u)$ and using the notation
\begin{eqnarray} \label{tildeA}
\tilde {\mathcal{A}}_0(s, u) = \left( 0,  \tilde{a_2}^0(s, u)\right)\,,\\\label{tildeAA}  \tilde{a_2}^0(s, u)= a_2^0\left(a(s) - u \dot{b}(s), b(s) +u \dot{a}(s)\right)
\end{eqnarray}
we perform the quadratic form (\ref{qtilde}) corresponding to operator $\tilde{H}_{\beta, \tilde{\Omega}}^{\mathcal{A}}(h),\, h\in \mathcal{H}^1(\tilde{\Omega})$ as follows

\begin{eqnarray} \nonumber
\tilde{q}_{\beta, \tilde{\Omega}}^\mathcal{A}(g)= \int_{\Omega_0} \left( \frac{1}{(1 + u \gamma)^2}\left|\left( \dot{a}
\frac{\partial}{\partial s} - (\dot{b} + u \ddot{a})\frac{\partial}{\partial u}\right) \left( \frac{h} {\sqrt{1 + u
\gamma}} \right) \right|^2 \right. \nonumber  \\\nonumber
 \left. \quad + \left| \left(\frac{i}{1+u \gamma} \left( \dot{b} \frac{\partial}{\partial s}
+ (\dot{a} - u \ddot{b}) \frac{\partial}{\partial u}\right) +
\mathcal{B} \tilde a_2^0 \right) \left( \frac{h}{\sqrt{1 + u \gamma}} \right)\right|^2 \right) (1 + u \gamma) \,d s\,d u- \beta \int_{\{u= 0\}_{s\in (s_1, s_2) }} |h|^2\,d s\\\nonumber- \beta \int_{\{u= f(s)\}_{s\in (s_1, s_2)}}\frac{1}{1+ f \gamma} \sqrt{\frac{(1+ f \gamma)^2+ \dot{f}^2}{1+ \dot{f}^2}} |h|^2\,d s- \beta \int_{\{s= s_1\}_{u\in (0, f(s_1))}} |h|^2\,d u- \beta \int_{\{s= s_2\}_{u\in (0, f(s_2))}} |h|^2\,d u\,.
\end{eqnarray}

Using (\ref{adot}) and (\ref{id}) the above expression performs
 \begin{gather}\nonumber
\tilde{q}_{\beta, \tilde{\Omega}}^\mathcal{A}(g)= \int_{\Omega_0} \biggl(\frac{1}{(1+ u\gamma)^2}\left|\frac
{\partial h}{\partial s}\right|^2 + i \mathcal{B}\frac{\dot{b}\tilde{a_2}^0}{1 + u \gamma} \left(\frac{\partial h}{\partial s} \overline
{h} - h \frac{\partial \overline{h}}{\partial s}\right) + \left|\frac{\partial h}{\partial u}\right|^2\\ \nonumber
+ i \mathcal{B}\frac{(\dot{a} - u \ddot{b})
\tilde{a_2}^0}{1 + u \gamma} \left(\frac{\partial h}{\partial u} \overline{h}
-\psi  {\frac{\partial\overline{h}}{\partial u}} \right) \\\nonumber
- \frac{u \dot{\gamma} }{2(1 + u \gamma)^3}\left(h
{\frac{\partial\overline{h}}{\partial s}} + \frac{\partial h}{\partial s} \overline{h}\right) - \frac{\gamma}{2(1 + u \gamma)} \left(h{\frac{\partial \overline{h}}{\partial u}} +
\frac{\partial h}{\partial u} \overline{h}\right)\\\nonumber
+ \left( \frac{u^2  \left(\dot{\gamma} \right)^2}
{4(1 + u \gamma)^4} + \frac{ \gamma^2}{4(1 + u
\gamma)^2}  + \mathcal{B}^2(\tilde{a_2}^0)^2\right) |h|^2\biggr)
\,d s\,d u\\\nonumber- \beta \int_{\{u= 0\}_{s\in(s_1, s_2)}} |h|^2\,d s- \beta \int_{\{u= f(s)\}_{s\in(s_1, s_2)}}\frac{1}{1+ f \gamma} \sqrt{\frac{(1+ f \gamma)^2+ \dot{f}^2}{1+ \dot{f}^2}} |h|^2\,d s\\\nonumber- \beta \int_{\{s= s_1\}_{u\in (0, f(s_1))}} |h|^2\,d u- \beta \int_{\{s= s_2\}_{u\in (0, f(s_2))}} |h|^2\,d u\,.
\end{gather}

We write the right-hand side of the preceding expression as a perturbation of the form $\tilde{q}_{\beta, \tilde{\Omega}}^A$, as follows:

\begin{equation}
\label{pert.}
\tilde{q}_{\beta, \tilde{\Omega}}^\mathcal{A}(g)= \breve{q}_{\beta, \Omega_0}^\mathcal{A}(h)- I(h)\,,
\end{equation}
where
\begin{eqnarray}\nonumber
\breve{q}_{\beta, \Omega_0}^\mathcal{A}(h)= \int_{\Omega_0} \left|i \frac{\partial h}{\partial s} + \mathcal{B} \dot{b} \tilde{a_2}^0 h\right|^2 + \left|i \frac{\partial h}{\partial u} +
 \mathcal{B} \dot{a} \tilde{a_2}^0 h\right|^2 \,d s\, d u\\\nonumber- \beta \int_{\{u= 0\}_{s\in (s_1, s_2)}}
|h|^2\,d s- \beta \int_{\{u= f(s)\}_{s\in (s_1, s_2)}}\frac{1}{1+ f \gamma} \sqrt{\frac{(1+ f \gamma)^2+ \dot{f}^2}{1+ \dot{f}^2}} |h|^2\,d s\\\label{q0}- \beta \int_{\{s= s_1\}_{u\in (0), f(s_1))}} |h|^2\,d u- \beta \int_{\{s= s_2\}_{u\in (0, f(s_2))}} |h|^2\,d u\,,\\
\nonumber
I(h)= \int_{\Omega_0} \biggl( \frac{2 u \gamma + u^2
\gamma^2}{(1 + u \gamma)^2} \left|\frac{\partial h}{\partial s}\right|^2 + i \frac{\mathcal{B}u \gamma \dot{b} \tilde{a_2}^0}{1+ u\gamma} \left(\frac{\partial h}{\partial s} \overline h
- h{\frac{\partial \overline{h}}{\partial s}}\right)\\\label{Ih}+ i \frac{\mathcal{B} u \tilde{a_2}^0}{1+ u \gamma} (
\gamma \dot{a}+ \ddot{b} ) \left( \frac{\partial h}{\partial u} \overline h -
h {\frac{\partial \overline{h}}{\partial u}} \right)  \\\nonumber
+ \frac{u  \dot{\gamma}}{2(1 + u \gamma)^3}
\left(h{\frac{\partial \overline{h}}{\partial s}}+ \frac{\partial h}{\partial s} \overline h
\right) + \frac{\gamma}{2(1 + u \gamma)} \left(h
{\frac{\partial \overline{h}}{\partial u}} + \frac{\partial h}{\partial u} \overline h \right)
- \left( \frac{u^2 \left( \dot{\gamma} \right)^2}
{4(1 + u \gamma)^4} + \frac{ \gamma^2} {4(1 + u
\gamma)^2} \right) |h|^2 \biggr) \,d s \,d u\,.
\end{eqnarray}

We now estimate $I(h)$. Since $u\in (0, f(s))$, equations (\ref{ab}) and (\ref{id}) yield
\begin{gather}\nonumber
| I(h)|\le \|f\|_\infty \left( \frac{2 +  \|f \gamma\|_\infty
}{(1 - \|f \gamma\|_\infty)^2} \right)\int_{\Omega_0}  |\gamma|\left|\frac{\partial h}{\partial s}\right|^2\,d s\,d u\\\nonumber + 2 \|f\|_\infty \left(\frac{\mathcal{B} \|\tilde{a_2}^0\|_\infty }{1 -  \|f \gamma\|_\infty}\right) \int_{\Omega_0} |\gamma| \left|\frac{\partial h}{\partial s}\right| 
|\psi|\,d s\,d u
+ 4 \|f\|_\infty \left(\frac{\mathcal{B} \|\tilde{a_2}^0\|_\infty}{1- \|f \gamma\|_\infty}\right)\int_{\Omega_0} |\gamma| \left|\frac{\partial h}{\partial u}\right| |h|\,d s\,d u 
\\\nonumber
+ \left(\frac{\|f\|_\infty}{(1 - \|f \gamma\|_\infty)^3} \right) \int_{\Omega_0} 
|\dot{\gamma}|\left|\frac{h}{\partial s}\right| |h|\,d s\,d u+ \left(\frac{1}{1 - \|f \gamma\|_\infty}\right) \int_{\Omega_0} 
|\gamma|\left|
\frac{\partial h}{\partial u}\right| |h|\,d s\,d u
\\\nonumber\le  \|f\|_\infty
\left(\frac{2 + \|f \gamma\|_\infty}{(1 - \|f \gamma\|_\infty)^2}+ \frac{\mathcal{B} \|\tilde{a_2}^0\|_\infty}{(1- \|f \gamma\|_\infty)}+ 
\frac{1}{2(1 - \|f \gamma\|_\infty)^3} \right)\times\\\nonumber\times \int_{\Omega_0}\mathrm{max}\{|\gamma|, |\dot{\gamma}|\}\left|\frac{\partial h}{\partial s}\right|^2\,d s\,d u\\\nonumber
+ \frac{1}{1 -  \|f \gamma\|_\infty} \left(2 \mathcal{B} \|\tilde{a_2}^0\|_\infty \|f\|_\infty+ \frac{1}{2}\right)
\int_{\Omega_0}|\gamma|
\left|\frac{\partial h}{\partial u}\right|^2\,d s\,d u\\\nonumber+
\left(  \|f\|_\infty \left( \frac{3 \mathcal{B} \|\tilde{a_2}^0\|_\infty}{(1- \|f \gamma\|_\infty)} +\frac{1}{2(1 - \|f \gamma\|_\infty)^3}\right)+ \frac{1}{2(1 - \|f \gamma\|_\infty)}\right)\times \\\nonumber\times
\int_{\Omega_0}\mathrm{max}\{|\gamma|, |\dot{\gamma}|\} |h|^2\,d s\,d u\,. 
\end{gather}

Let us denote
\begin{equation}\label{alpha1}
\alpha_1:=\mathrm{max}\{\tau_1, \,\tau_2\}\,,
\end{equation}
where \begin{eqnarray*} \tau_1=   \|f\|_\infty
\left(\frac{2 +  \|f \gamma\|_\infty}{(1 -  \|f \gamma\|_\infty)^2}+ \frac{\mathcal{B} \|\tilde{a_2}^0\|_\infty}{(1- \|f \gamma\|_\infty)
}+ 
\frac{1}{2(1 - \|f \gamma\|_\infty)^3} \right) \,,
\\
\tau_2= \frac{1}{1 - \|f \gamma\|_\infty} \left( 2 \mathcal{B} \|\tilde{a_2}^0\|_\infty \|f\|_\infty+ \frac{1}{2}\right)
\,,
\end{eqnarray*}
and
\begin{equation}\label{alpha2}
\alpha_2:= 
 \|f\|_\infty \left( \frac{3 \mathcal{B} \|\tilde{a_2}^0\|_\infty}{(1- \|f\|_\infty \gamma\|_\infty)} +\frac{1}{2(1 - \|f \gamma\|_\infty)^3}\right)+ \frac{1}{2(1 - \|f \gamma\|_\infty)}\,.
\end{equation}

Hence the above inequality implies 
$$ |I(h)|\le \alpha_1\int_{\Omega_0}\mathrm{max}\{|\gamma|, |\dot{\gamma}|\} |\nabla h|^2\,d s\,d u+
 \alpha_2\int_{\Omega_0}\mathrm{max}\{|\gamma|, |\dot{\gamma}|\} |h|^2\,d s\,d u\,.
$$

Combining (\ref{pert.}) with the preceding estimate, we obtain
\begin{equation}
\label{est..}
\tilde{q}_{\beta, \tilde{\Omega}}^\mathcal{A}(g)\ge \breve{q}_{\beta, \Omega_0}^A(h)- \alpha_1\int_{\Omega_0}\mathrm{max}\{|\gamma|, |\dot{\gamma}|\}
|\nabla h|^2\,d s\,d u-
 \alpha_2\int_{\Omega_0}\mathrm{max}\{|\gamma|, |\dot{\gamma}|\} |h|^2\,d s\,d u\,.
\end{equation}

For any magnetic potential $\hat{A}= (\hat{a_1}, \hat{a_2})$, the following pointwise inequality holds:
\begin{equation}\label{aux.}
\left|\nabla h\right|^2\le 2\left|i\nabla h+ \hat{A}h\right|^2 + 2\left(\hat{a_1}^2+ \hat{a_2}^2\right)
|h|^2\,, \quad h\in \mathcal{H}^1(\Omega_0)\,.
\end{equation}

In view of this fact, (\ref{est..}) and expression for $\breve{q}_{\beta, \Omega_0}^{\mathcal{A}}$ (\ref{q0}) we get
\begin{eqnarray}\nonumber
\tilde{q}_{\beta, \tilde{\Omega}}^\mathcal{A}(g)\ge \int_{\Omega_0} \left(1- 2\alpha_1\mathrm{max}\{|\gamma|, |\dot{\gamma}|\}\right)\left|i \nabla h+ \mathcal{B}\breve{A}^0 h\right|^2\,d s\,d u\\\nonumber- 
\int_{\Omega_0}\mathrm{max}\{|\gamma|, |\dot{\gamma}|\}\left(2\alpha_1 \mathcal{B}^2\|\tilde{a_2}^0\|_\infty^2+ \alpha_2 \right)|h|^2\,d s\,d u\\\nonumber- \beta \int_{\{u= 0\}_{s\in (s_1, s_2)}}
|h|^2\,d s- \beta \int_{\{u= f(s)\}_{s\in (s_1, s_2)}}\frac{1}{1+ f \gamma} \sqrt{\frac{(1+ f \gamma)^2+ \dot{f}^2}{1+ \dot{f}^2}} |h|^2\,d s\\\label{bq}- \beta \int_{\{s= s_1\}_{u\in (0, f(s_1))}} |h|^2\,d u- \beta \int_{\{s= s_2\}_{u\in (0, f(s_2))}} |h|^2\,d u\,,\end{eqnarray}
where
\begin{equation}\label{breve} \breve{A}^0= (\dot{b} \tilde{a_2}^0,  \dot{a} \tilde{a_2}^0)\,.
\end{equation}
 
In view of (\ref{alpha1}) and (\ref{alpha2}) it can be checked that $\alpha_1, \alpha_2= \mathcal{O}(\mathcal{B})$ for $\mathcal{B}$ large enough. Then by choosing $\gamma$ such that
\begin{equation}\label{first cond.}
\mathrm{max}\{|\gamma|, |\dot{\gamma}|\} \alpha_1< \frac{1}{4}
\end{equation}
we have 

\begin{eqnarray}\nonumber
\tilde{q}_{\beta, \tilde{\Omega}}^{\mathcal{A}}(g)\ge \frac{1}{2}\int_{\Omega_0}\left|i \nabla h+ \mathcal{B}\breve{A}^0 h\right|^2\,d s\,d u\\\nonumber- 
\int_{\Omega_0}\mathrm{max}\{|\gamma|, |\dot{\gamma}|\}\left(2\alpha_1 \mathcal{B}^2 \|\tilde{a_2}^0\|_\infty^2+ \alpha_2 \right)|h|^2\,d s\,d u\\\nonumber- \beta \int_{\{u= 0\}_{s\in (s_1, s_2)}}
|h|^2\,d s- \beta \int_{\{u= f(s)\}_{s\in (s_1, s_2)}}\frac{1}{1+ f \gamma} \sqrt{\frac{(1+ f \gamma)^2+ \dot{f}^2}{1+ \dot{f}^2}} |h|^2\,d s\\\label{new.op}- \beta \int_{\{s= s_1\}_{u\in (0, f(s_1))}} |h|^2\,d u- \beta \int_{\{s= s_2\}_{u\in (0, f(s_2))}} |h|^2\,d u\,.
\end{eqnarray}

In view of (\ref{A0}), (\ref{tildeAA}) and (\ref{breve}) we have that $\breve{A}^0=(\dot{b}(a -u \dot{b}), \dot{a}(a- u \dot{b}))$ on $\Omega_0$. 

Further we need to introduce the following magnetic potential $A_1^0= (0, s)$ and to mention the simple inequality  
\begin{eqnarray*}
\int_{\Omega_0}|\breve{A}^0- A_1^0|^2|h|^2\,d s\,d u
\\\le 2\left(\|\dot{b}\|_\infty^2 (\|a\|_\infty^2+ \|f \dot{b}\|_\infty^2)+\|\dot{a} a- s\|_\infty^2+ \|f \dot{a} \|_\infty^2 \|\dot{b}\|_\infty^2\right)
\int_{\Omega_0}|h|^2\,d s\,d u\,,
\end{eqnarray*}
from which immediately implies that

\begin{eqnarray}\nonumber
\int_{\Omega_0}\left|i \nabla h+ \mathcal{B}\breve{A}^0 h\right|^2\,d s\,d u\ge \frac{1}{2}\int_{\Omega_0}\left|i \nabla h+ \mathcal{B}A_1^0 h\right|^2\,d s\,d u\\\nonumber- 2\mathcal{B}^2 \left(\|\dot{b}\|_\infty^2 (\|a\|_\infty^2+ \|f \dot{b}\|_\infty^2)+\|\dot{a} a- s\|_\infty^2+ \|f \dot{a} \|_\infty^2 \|\dot{b}\|_\infty^2\right)
\int_{\Omega_0}|h|^2\,d s\,d u\,.
\end{eqnarray}

Hence (\ref{bq}) implies
\begin{eqnarray}\nonumber
\tilde{q}_{\beta, \tilde{\Omega}}^{\mathcal{A}}(g)\ge \frac{1}{2}\int_{\Omega_0}\left|i \nabla h+ \mathcal{B}A_1^0 h\right|^2\,d s\,d u\\\nonumber - \frac{\beta \sqrt{(1+ \|f \gamma\|_\infty)^2+ \|\dot{f}\|_\infty^2}}{1- \|f \gamma\|} \biggl(\int_{\{u= 0\}_{s\in (s_1, s_2)}}
|h|^2\,d s-  \int_{\{u= f(s)\}_{s\in (s_1, s_2)}} |h|^2\,d s\\\nonumber- \int_{\{s= s_1\}_{u\in (0, f(s_1))}} |h|^2\,d u- \int_{\{s= s_2\}_{u\in (0, f(s_2))}} |h|^2\,d u\biggr)\\\nonumber- 
\int_{\Omega_0}\mathrm{max}\{|\gamma|, |\dot{\gamma}|\}\left(2\alpha_1 \mathcal{B}^2 \|\tilde{a_2}^0\|_\infty^2+ \alpha_2 \right)|h|^2\,d s\,d u\\\label{fin}- 2\mathcal{B}^2 \left(\|\dot{b}\|_\infty^2 (\|a\|_\infty^2+ \|f \dot{b}\|_\infty^2)+\|\dot{a} a- s\|_\infty^2+ \|f \dot{a} \|_\infty^2 \|\dot{b}\|_\infty^2\right)
\int_{\Omega_0}|h|^2\,d s\,d u\,.
\end{eqnarray}

To continue the proof we need the lower bound for the magnetic Robin Laplacian eigenvalues. Let $\omega\subset\mathbb{R}^2$ be an open set. We will assume that the boundary of $\omega$ is $C^3$  smooth, compact and consists  of a finite number of connected
components. Assume that $\tilde{\beta} \in \mathbb{R},\, \mathcal{B}\ge 0$ and let $\tilde{H}^\omega_{\tilde{\beta}}(\mathcal{B})$ be the self-adjoint operator in $L^2(\omega)$ corresponding to quadratic form
$$
\int_\omega\left|i \nabla v+ \mathcal{B} \overline{A}_0 v\right|^2\,d x\,d y- \tilde{\beta}\int_{\partial \omega}|v|^2\,d s\,,v\in \mathcal{H}^1(\omega)\,,
$$
where $\overline{A}_0= (0, x)$ is the magnetic potential corresponding to constant magnetic field: $\mathrm{curl}(\overline{A}_0)= 1$.
The following theorem takes place-- \cite{K16}:

\begin{theorem}\label{lemma1} Let $\alpha\in \left(\frac{1}{2}, 1\right)$, $0< \tilde{c}_1 <\tilde{c}_2$ and $\tilde{\beta}_0> 0$. Suppose that
$$\tilde{\beta}> \tilde{\beta}_0\quad\text{and}\quad \tilde{c}_1 \tilde{\beta}^{\frac{1}{1-\alpha}}\le b\le \tilde{c}_2 \tilde{\beta}^{\frac{1}{1-\alpha}}\,.$$ Let $\mu_1(\tilde{\beta}, b)$ be the ground state eigenvalue of  $\tilde{H}_{\tilde{\beta}}(b)$. Then 
\begin{enumerate}
\item If $\alpha> \frac{1}{2}$,  the ground state eigenvalue satisfies, as $\tilde{\beta}\to\infty$,
$$\tilde\mu_1(\tilde{\beta}, b)=\Theta_0 b+ b o(1)\,.$$
where $\Theta_0\in(0,1)$ is a universal constant.
\item If $\alpha= \frac{1}{2}$,  the ground state eigenvalue satisfies, as $\tilde{\beta}\to\infty$,
$$\tilde\mu_1(\tilde{\beta}, b)=b \Theta_0 (\tilde{\beta} b^{-1/2})+ b o(1)\,.$$
\end{enumerate}
\end{theorem}

Let us return to (\ref{fin}). Since $\Omega_0$ is $C^3$- smooth domain then with the notations $\tilde{\beta}=\frac{2\beta \sqrt{(1+ \|f \gamma\|_\infty)^2+ \|\dot{f}\|_\infty^2}}{1-\|f \gamma\|_\infty}$\,,\newline \noindent $\tilde{c}_1=\frac{c_1 (1-\|f \gamma\|_\infty)^2}{4 ((1+ \|f \gamma\|_\infty)^2+ \|\dot{f}\|_\infty^2)}\,,\, \tilde{c}_2=\frac{c_2 (1-\|f \gamma\|_\infty)^2}{4 ((1+ \|f \gamma\|_\infty)^2+ \|\dot{f}\|_\infty^2)}$ and $\omega= \Omega_0$ we are able to use the statement of Theorem \ref{lemma1}. Hence for $\beta$ large enough
\begin{eqnarray}\nonumber
\tilde{q}_{\beta, \tilde{\Omega}}^{\mathcal{A}}(g)\ge \frac{\kappa}{2} \mathcal{B} \int_{\Omega_0}|h|^2\,d s\,d u\\\nonumber- 
\int_{\Omega_0}\mathrm{max}\{|\gamma|, |\dot{\gamma}|\}\left(2\alpha_1 \mathcal{B}^2 \|\tilde{a_2}^0\|_\infty^2+ \alpha_2 \right)|h|^2\,d s\,d u\\\nonumber- 2\mathcal{B}^2 \left(\|\dot{b}\|_\infty^2 (\|a\|_\infty^2+ \|f \dot{b}\|_\infty^2)+\|\dot{a} a- s\|_\infty^2+ \|f \dot{a} \|_\infty^2 \|\dot{b}\|_\infty^2\right)
\int_{\Omega_0}|h|^2\,d s\,d u\,,
\end{eqnarray}
where $\kappa= \mathrm{max}\left\{\frac{ \Theta_0}{2},\, \frac{ \Theta_0}{2 \sqrt{\tilde{c}_2}}\right\}$.

The above inequality establishes that if $\|\dot{b}\|_\infty,\, \|\dot{a} a- s\|_\infty$ and $\|\gamma\|_\infty, \,\|\dot{\gamma}\|_\infty$ are small enough such that

\begin{eqnarray}\nonumber
2\mathcal{B}^2 \left(\|\dot{b}\|_\infty^2 (\|a\|_\infty^2+ \|f \dot{b}\|_\infty^2)+\|\dot{a} a- s\|_\infty^2+ \|f \dot{a} \|_\infty^2 \|\dot{b}\|_\infty^2\right)< \frac{\kappa \mathcal{B}}{4}\,,
\\\label{1case}
\mathrm{max}\{|\gamma|, |\dot{\gamma}|\}\left(2\alpha_1 \mathcal{B}^2 \|\tilde{a_2}^0\|_\infty^2+ \alpha_2 \right)<  \frac{\kappa \mathcal{B}}{4}
\end{eqnarray}
then
$$ \tilde{q}_{\beta, \tilde{\Omega}}^{\mathcal{A}}(g) \ge 0\,.
$$

Employing (\ref{reconstr.}) one can easily verify that the quantities $\|\dot{b}\|_\infty$ and $\|\dot{a} a- s\|_\infty$ can be made small enough in order to guarantee the validity of (\ref{1case}) by choosing $L^\infty$ norm of $\gamma$ bounded by a sufficiently small number depending on $\mathcal{B}$. The same is true for the $L^\infty$ norms of $\gamma$ and $\dot{\gamma}$ in the necessary condition (\ref{first cond.}) obtained previously. Thus, the proof of the theorem is complete with sufficiently large $\beta_0$ and sufficiently small  $\varepsilon_0(\mathcal{B})$.

\section{Stability of the essential spectrum}\label{stability}
\setcounter{equation}{0}

In this section we show that the local perturbation of the curve $\Gamma$ does not change the essential spectrum. The following theorem takes place:

\begin{theorem}
Suppose assumptions stated in section Introduction. Then the essential spectrum of the operator $H^{\mathcal{A}}_{\beta, \Omega}$ coincides with the half-line $\left[-\beta^2, \infty\right)$.
\end{theorem}

\begin{proof}

To prove that any non-negative number $\mu\ge -\beta^2$ belongs to the essential
spectrum of $H^{\mathcal{A}}_{\beta, \Omega}$, we will use Weyl's criterion
\cite[Thm.~VII.12]{RS81}: we have to find a sequence
$\{\phi_n\}_{n=1}^\infty\subset D(H^{\mathcal{A}}_{\beta, \Omega})$ of unit vectors,
$\|\phi_n\|=1$, which converges weakly to zero and
$$
\|H^{\mathcal{A}}_{\beta}\phi_n-\mu\phi_n\|\to 0 \qquad\text{as}\quad n\to\infty
$$
holds. 

First, let us rewrite operator $H^{\mathcal{A}}_{\beta, \Omega}$ in the following form

$$
H^{\mathcal{A}}_{\beta, \Omega}= -\Delta+ 2i  a_2^0 \frac{\partial}{\partial y} + \left(a_2^0\right)^2.
$$
Then
\begin{eqnarray}\nonumber
\int_\Omega\left|H_{\beta, \Omega}^{\mathcal{A}}\phi_n- \mu \phi_n\right|^2\,d x\,d y\\\nonumber=
\int_\Omega\biggl|-\Delta \phi_n+ 2i a_2^0 \frac{\partial\phi_n}{\partial y}+ \left(a_2^0\right)^2 \phi_n- \mu\phi_n\biggr|^2\,d x\,d y
\end{eqnarray}
and 
\begin{eqnarray}\nonumber
\int_\Omega\left|H_{\beta, \Omega}^{\mathcal{A}}\phi_n- \mu \phi_n\right|^2\,d x\,d y\\\label{Weyl}\le
4\int_\Omega\left|-\Delta \phi_n- \mu\phi_n\right|^2\,d x\,d y+ 16\int_\Omega \left(a_2^0\right)^2\left|\frac{\partial\phi_n}{\partial y}\right|^2\,d x\,d y+ 4\int_\Omega \left(a_2^0\right)^4 |\phi_n|^2\,d x\,d y\,.
\end{eqnarray}

For each $\mu= -\beta^2+ k^2,\,k\in\mathbb{Z}$, we will separately estimate each integral in (\ref{Weyl}).
We will construct the Weyl sequence $\phi_n\in C_0^\infty(\Omega)$ as follows 
\begin{equation}\label{psi}
\phi_n(x, y)= \sqrt{\frac{2 
\beta}{n}}\chi\left(\frac{x}{n}\right) e^{i k x} e^{-\beta y},\quad n\in \mathbb{N},
\end{equation}
where $\chi\in C_0^\infty(\mathbb{R})$ is a smooth function with support in the interval $(1, 2)$ with $L^2$ norm equal to one.

We have 
\begin{eqnarray}\nonumber
\int_{\Omega}\left|H_{\beta, \Omega}^{\mathcal{A}}\phi_n- \mu\phi_n\right|^2\,d x\,d y\\\nonumber= \frac{8 \beta}{n}\int_{-1}^1 \int_0^\infty\left| -\frac{1}{n^2}\ddot{\chi}\left(\frac{x}{n}\right)- \frac{2 i k}{n}\dot{\chi}\left(\frac{x}{n}\right) + k^2\chi\left(\frac{x}{n}\right) -
\beta^2 \chi\left(\frac{x}{n}\right)- \mu \chi\left(\frac{x}{n}\right) \right|^2\,e^{-2\beta y}\,d x\,d y\\\nonumber+ \frac{32 \beta^3}{n}\int_{-1}^1 \int_0^\infty \left(a_2^0(x)\right) ^2\left|\chi\left(\frac{x}{n}\right)\right|^2\,e^{-2\beta y}\,\,d x\,d y+ \frac{8 \beta}{n}\int_{-1}^1 \int_0^\infty\left(a_2^0(x)\right)^4 \left|\chi\left(\frac{x}{n}\right)\right|^2\,\,e^{-2\beta y}\,d x\,d y\\\nonumber = \frac{8 \beta}{n}\int_{-1}^1 \int_0^\infty \left| -\frac{1}{n^2}\ddot{\chi}\left(\frac{x}{n}\right) - \frac{2 i k}{n}\dot{\chi}\left(\frac{x}{n}\right)\right|^2\,e^{-2\beta y}\,d x\,d y\\\nonumber+ \frac{32 \beta^3}{n}\int_{-1}^1 \int_0^\infty \left(a_2^0(x)\right)^2\left|\chi\left(\frac{x}{n}\right)\right|^2\,e^{-2\beta y}\,d x\,d y+ \frac{8 \beta}{n}\int_{-1}^1 \int_0^\infty\left(a_2^0(x)\right)^4 \left|\chi\left(\frac{x}{n}\right)\right|^2\,e^{-2\beta y}\,d x\,d y\\\nonumber\le \frac{8}{n^4}\int_1^2|\ddot{\chi}(x)|^2\,d x+ \frac{32 k^2}{n^2}\int_1^2|\dot{\chi}(x)|^2\,d x+ 16 \beta^2\left\|a_2^0\right\|_{L^\infty(n, 2n)}^2 + 4\left\|a_2^0\right\|_{L^\infty(n, 2n)}^4\,.
\end{eqnarray}

Finally in view of assumption (\ref{A0}) the latter completes the theorem.
\end{proof}

\subsection*{Contributions}
Authors have been discussing and working together on the manuscript, contributing equally to the content, presentation, and reviewing the manuscript.
\bigskip

\subsection*{Data Availability}
No datasets were generated or analysed during the current study.
\bigskip

\subsection*{Competing interests}
The authors declare no competing interests.

\section{Acknowledgements}

Y.Z. was co-funded by the Czech Science Foundation (GA\v{C}R), Grant No.~25-16847S.  It was also supported by the University of Ostrava, Grant No.~SGS05/P\v{R}F/2026.


\begin{thebibliography}{10}


\bibitem{BBS24} J.~Bory-Reyes, D.~Barseghyan, B.~ Schneider: Magnetic Schr\"{o}dinger operator with the potential supported in a curved two-dimensional strip,  Mediterranean Journal of Mathematics 21(3) (2024), 1--15.

\bibitem{BE21} D.~Barseghyan, P.~Exner: Magnetic field influence on the discrete spectrum of locally deformed leaky wires, Reports on Mathematical Physics 88 (1) (2015), 47--57.

\bibitem{BEGK01} D.~ Borisov, P. ~Exner, R.R. ~Gadyl’shin and D.~Krej\v ci\v
 r\'{\i}k: Bound states in weakly deformed strips and layers, Ann. Henri  Poincar\'e 2 (2001), 553--572.

\bibitem{BGRS97}W.~ Bulla, F. Gesztesy, W. ~Renger and B.~ Simon: Weakly coupled bound states in quantum waveguides, Proc. Amer. Math. Soc. 125 (1997), no. 5, 1487--1495.

\bibitem{DE95} P. ~Duclos and P. ~Exner: Curvature-induced bound states in quantum waveguides in two and three dimensions, Rev. Math. Phys. 7 (1995), 73--102.

\bibitem{EM14} P.~Exner, M.~ Minakov, Curvature-induced bound states in Robin waveguides and their asymptotical properties,  J.Math.Phys. 55 (2014) 122101.

\bibitem{ES89} P.~Exner, P.~\v{S}eba,
Bound states in curved quantum wavequides, J. Math. Phys. 30 (1989), 2574--2580.

\bibitem{GJ92}
J. ~Goldstone and R.L.~ Jaffe: Bound states in twisting tubes, Phys. Rev. B45 (1992), 14100--14107.
\bibitem{EK15} P.~Exner, H.~Kova\v{r}\'{\i}k: Quantum Waveguides, Springer International, Heidelberg 2015.


\bibitem{K16} A.~ Kachmar, Diamagnetism versus Robin condition and concentration of ground states. Asymptot. Anal. 98 (2016), no. 4, 341--375.  

\bibitem{RS81} M.~ Reed, B. ~Simon: Methods of Modern Mathematical Physics, I. Functional Analysis, II. Fourier Analysis, IV. Analysis of Operators. Self-Adjointness, Academic Press, New York 1981, 1975, 1978.
 
\end{thebibliography}
\end{document}